\pdfoutput = 1

\documentclass[11pt]{article}
\usepackage{fullpage}

\usepackage[utf8]{inputenc} 
\usepackage[T1]{fontenc}    
\usepackage{hyperref}       
\usepackage{url}            
\usepackage{booktabs}       
\usepackage{amsfonts}       
\usepackage{nicefrac}       
\usepackage{microtype}      
\usepackage{xcolor}         

\usepackage[pdftex]{graphicx}
\usepackage{subcaption}

\usepackage{amsfonts}       
\usepackage{amssymb}
\usepackage{amsmath}
\usepackage{amsthm}
\usepackage{mathtools}
\usepackage{multicol}
\usepackage{multirow}

\usepackage{color}

\usepackage{array}
\usepackage{algorithm}
\usepackage{algorithmic}

\usepackage[shortlabels]{enumitem}

\usepackage[symbol]{footmisc} 

\def\M{\mathcal{M}}

\def\D{\mathrm{D}}

\def\sR{{\mathbb R}}
\def\sS{{\mathbb S}}
\def\sP{{\mathbb P}}
\def\sE{{\mathbb E}}

\def\gA{{\mathcal{A}}}
\def\gB{{\mathcal{B}}}
\def\gD{{\mathcal{D}}}

\def\gN{{\mathcal{N}}}

\def\gW{{\mathcal{W}}}

\def\gS{{\mathcal{S}}}

\def\gR{{\mathcal{R}}}
\def\gL{{\mathcal{L}}}

\def\grad{{\mathrm{grad}}}

\newcommand{\trace}{\mathrm{tr}}

\newtheorem{theorem}{Theorem}
\newtheorem{lemma}{Lemma}
\newtheorem{proposition}{Proposition}

\theoremstyle{definition}
\newtheorem{definition}{Definition}

\newtheorem{remark}{Remark}

\newcolumntype{P}[1]{>{\centering\arraybackslash}p{#1}}

\title{\textbf{Differentially private Riemannian optimization}}
\author{Andi Han\footnote{University of Sydney 
  (\texttt{andi.han@sydney.edu.au}, \texttt{junbin.gao@sydney.edu.au}).}
\and Bamdev Mishra\footnote{Microsoft India.
  (\texttt{bamdevm@microsoft.com}, \texttt{pratik.jawanpuria@microsoft.com}).}
 \and Pratik Jawanpuria\footnotemark[2] 
\and Junbin Gao\footnotemark[1]}
\date{}

\begin{document}

\maketitle

\begin{abstract}
    In this paper, we study the differentially private empirical risk minimization problem where the parameter is constrained to a Riemannian manifold. We introduce a framework of differentially private Riemannian optimization by adding noise to the Riemannian gradient on the tangent space. The noise follows a Gaussian distribution intrinsically defined with respect to the Riemannian metric. We adapt the Gaussian mechanism from the Euclidean space to the tangent space compatible to such generalized Gaussian distribution. We show that this strategy presents a simple analysis as compared to directly adding noise on the manifold. We further show privacy guarantees of the proposed differentially private Riemannian (stochastic) gradient descent using an extension of the moments accountant technique. Additionally, we prove utility guarantees under geodesic (strongly) convex, general nonconvex objectives as well as under the Riemannian Polyak–\L{}ojasiewicz condition. We show the efficacy of the proposed framework in several applications.
\end{abstract}

\section{Introduction}

With the ever-increasing complication of statistics and machine learning models, data privacy has become a primary concern as it becomes increasingly difficult to safeguard the potential disclosure of private information during model training. Differential privacy \cite{dwork2006calibrating,dwork2014algorithmic} provides a framework for quantifying the privacy loss as well as for designing algorithms with privacy-preserving guarantees. 

Many problems in machine learning fall under the paradigm of empirical risk minimization (ERM), where the loss is expressed as $\frac{1}{n}\sum_{i=1}^n f(w; z_i)$ with independent and identically distributed (i.i.d) samples $z_1,\ldots, z_n$ drawn from a data distribution $\gD$. Differentially private ERM, originally studied in \cite{chaudhuri2011differentially}, aims to safeguard the privacy disclosure of the samples in the solution $w^*$. There exist many approaches to achieve this goal. The first class of methods is to perturb the output of a non-differentially private algorithm by adding a Laplace or Gaussian noise \cite{chaudhuri2008privacy,chaudhuri2011differentially,zhang2017efficient}. Another approach considers adding a linear random perturbation term to the objective and is known as objective perturbation \cite{chaudhuri2008privacy,chaudhuri2011differentially,kifer2012private,iyengar2019towards,bassily2021differentially}. The third type of approach is to inject noise to gradient based algorithms at each iteration \cite{bassily2014private,wang2017differentially,wang2019differentially,abadi2016deep,bassily2019private}. In addition, there also exist various specialized methods for problems such as linear regression and statistics estimation \cite{dwork2009differential,nissim2007smooth,amin2019differentially,kamath2019privately,biswas2020coinpress}.

Among all the aforementioned approaches, gradient perturbation receives the most attention due to its generality for arbitrary loss functions and scalability to large datasets. Furthermore, it only requires to bound the sensitivity of the gradients computed at each iteration, rather than the entire process. There is an ongoing line of work that aims to improve the utility of the gradient perturbed algorithms while maintaining the same amount of privacy budget. Such improvements have been seen under (strongly) convex losses \cite{wang2017differentially,iyengar2019towards,yu2021gradient,bassily2019private,kuru2022differentially,asi2021private}, nonconvex losses \cite{wang2017differentially,zhang2017efficient,wang2019differentially,wang2019efficient}, and also structured losses such as satisfying Polyak-\L{}ojasiewicz condition \cite{wang2017differentially}.

In this paper, we consider following ERM problem in a differentially private setting where the parameter is constrained to lie on a Riemannian manifold, i.e.,
\begin{equation}
    \min_{w \in \M} \Big\{ F(w) = \frac{1}{n} \sum_{i=1}^n f_i(w) = \frac{1}{n} \sum_{i=1}^n f( w ; z_i ) \Big\}, \label{main_problem}
\end{equation}
where $\M$ is a $d$-dimensional Riemannian manifold and $f: \M \times \gD \xrightarrow{} \sR$ is a loss function over samples. Riemannian manifolds commonly occur in statistics and machine learning where the parameters naturally possess additional nonlinear structure, such as orthogonality \cite{absil2009optimization}, positive definiteness \cite{bhatia2009positive}, unit norm, and hyperbolic \cite{boumal2020introduction} among others. Popular applications involving above manifold structures include matrix completion~\cite{boumalrtrmc, han2021improved}, metric learning, covariance estimation~\cite{han2021riemannian}, principal component analysis~\cite{absil2007trust}, and taxonomy embedding \cite{nickel2018learning}, to name a few. 


While few recent works address specific Riemannian optimization problems under differential privacy, such as private Fr\'echet mean computation \cite{reimherr2021differential}, there exists no systematic study of general purpose strategies to guarantee differential privacy for  \eqref{main_problem} on Riemannian manifolds. On the other hand, differentially private (non-Riemannian) approaches have been studied for ERM problems with specific constraints \cite{chaudhuri2013near,bassily2014private,jiang2016wishart,abadi2016deep,bassily2019private,maunu2022stochastic}. Such approaches typically employ the projected gradient algorithm, i.e., taking the gradient step and adding the noise in the Euclidean space, and then projecting onto the constraint set. Extending such a strategy to Riemannian manifolds may result in looser sensitivity and utility bounds scaling poorly with the dimension of the ambient space, which can be much larger than the intrinsic manifold dimension \cite{reimherr2021differential}.


\paragraph{Contributions.} 
In this work, we propose a general framework via Riemannian optimization to achieve differential privacy for \eqref{main_problem} by adding noise to the Riemannian gradient adhering to the Riemannian metric (an inner product formally defined in Section \ref{prelim_sect}). To this end, we generalize the Gaussian mechanism to the tangent space of Riemannian manifolds and also adapt the moments accountant technique to trace the privacy loss. We study the privacy guarantees of the differentially private Riemannian (stochastic) gradient descent method. Additionally, we show its utility guarantees for a variety of interesting function classes on Riemannian manifolds, including geodesic (strongly) convex, general nonconvex functions, and functions satisfying Riemannian Polyak-\L{}ojasiewicz (PL) conditions. 
A summary of utility bounds with our proposed framework is in Table \ref{table_main_result}. 
In addition, we show that the projected gradient methods for \eqref{main_problem} ignore the intrinsic geometry when taking the update and thus hamper the utility under structured loss functions on manifolds (discussed in Section~\ref{diff_private_riemannian_opt_sect}). 
Finally, we provide illustrating examples and empirical results in Section \ref{application_sect}. 

\begin{table}[t]
\begin{center}
\renewcommand{\arraystretch}{1.4}
{\small 
\caption{Utility guarantees of proposed $(\epsilon, \delta)$-differentially private Riemannian (stochastic) gradient descent under different function classes. $L_0, L_1$ are geodesic Lipschitz and smoothness constants of $f$ and $\beta, \tau$ are constants of geodesic strong convexity and Riemannian PL condition of $F$ respectively. $n$ is the size of the dataset. $d$ is the intrinsic dimension of the manifold. $\varsigma$ is the curvature constant of the domain (defined in Lemma \ref{curvature_lemma}). All utility bounds are measured in the expected empirical excess risk $\sE[F(w^{\rm priv})] - F(w^*)$ where $w^*$ is a global minimizer except for the general nonconvex case where the bound is measured in $\sE \| F(w^{\rm priv}) \|^2_{w^{\rm priv}}$. The bounds hide dependence on $c_l$ (lower bound on the metric tensor) and $D_\gW$ (diameter bound of the domain). Please also refer to Section \ref{convergence_sect}.} 
\label{table_main_result}
\begin{tabular}{@{}P{0.20\textwidth}P{0.24\textwidth}P{0.25\textwidth}P{0.22\textwidth}@{}}
\toprule
Geodesic convex & Geodesic strongly convex  & Riemannian PL condition & General nonconvex \\
\midrule
$O\Big( \frac{\sqrt{d \log(1/\delta) \varsigma} L_0 }{n\epsilon} \Big)$ & $O \Big( \frac{\beta^{-1} d \log(1/\delta) \varsigma  L_0^2}{ n^2 \epsilon^2} \Big)$  & $O\Big( \frac{\tau^{-1} d \log(n) \log(1/\delta)L_0^2}{n^2\epsilon^2} \Big)$  & $O \Big( \frac{\sqrt{d L_1 \log(1/\delta) } L_0}{n\epsilon}   \Big)$ \\
\bottomrule
\end{tabular}
}
\end{center}
\end{table}


\section{Preliminaries and related work}
\label{prelim_sect}

\paragraph{Riemannian geometry.} A Riemannian manifold $\M$ of dimension $d$ is a smooth manifold with an inner product structure $\langle \cdot , \cdot \rangle_w$ (i.e., a Riemannian metric) on every tangent space $T_w\M$. Given an orthonormal basis $(\partial_1, \ldots, \partial_d)$ for $T_w\M$, the metric can be expressed as a (symmetric positive definite) matrix $G_w$ and the inner product can be written as $\langle \xi , \zeta \rangle_w = \vec{\xi}^\top G_w \vec{\zeta}$ where $\vec{\xi}, \vec{\zeta} \in \sR^d$ are the vectorization of tangent vectors $\xi, \zeta \in T_w\M$ in the normal coordinate system.  An induced norm is defined as $\| \xi \|_w = \sqrt{\langle \xi, \xi\rangle_w}$ for any $\xi \in T_w\M$. A geodesic $\gamma: [0,1] \xrightarrow{} \M$ is a locally distance minimizing curve on the manifold with zero acceleration. For any $\xi \in T_w\M$, the exponential map is defined as ${\rm Exp}_w(\xi) = \gamma(1)$ where $\gamma(0) = w$ and $\gamma'(0) = \xi$. If, between 
two points $w, w'\in \M$, there exists a unique geodesic connecting them, the exponential map has a smooth inverse and the Riemannian distance is given by ${\rm dist}(w, w') = \| {\rm Exp}_w^{-1}(w') \|_w = \| {\rm Exp}_{w'}^{-1}(w) \|_{w'}$. We call a neighbourhood $\gW$ totally normal if for any two points, the exponential map is invertible. The Riemannian gradient of a real-valued function, denoted as $\grad F(w)$, is a tangent vector that satisfies, for any $\xi \in T_w\M$, it holds that $\langle \grad F(w), \xi \rangle_w = \D_\xi F(w) = \langle \nabla F(w), \xi \rangle_2$ where $\D_\xi F(w)$ is the directional derivative of $F(w)$ along $\xi$ and $\nabla F(w)$ is the Euclidean gradient.

\paragraph{Riemannian optimization.}
Under non-private settings, Riemannian optimization \cite{absil2009optimization,boumal2020introduction} provides a class of methods to efficiently solve problem \eqref{main_problem} for arbitrary loss functions by treating the constrained problem as unconstrained problem over manifolds. Given the Riemannian gradient, Riemannian steepest descent \cite{udriste2013convex} takes a gradient update via the Exponential map so that the iterates stay on the manifold, i.e., ${\rm Exp}_{w}(- \eta \, \grad F(w))$ for some stepsize $\eta$. Other more advanced solvers include Riemannian conjugate gradient \cite{ring2012optimization}, trust region methods \cite{absil2007trust}, as well as many recent stochastic optimizers \cite{bonnabel2013stochastic,zhang2016riemannian,sato2019riemannian,kasai2018riemannian,hanmomentum2021,han2021improved}.

\paragraph{Function classes on Riemannian manifolds.} The notion of Lipschitz continuity has been generalized to Riemannian manifolds \cite{boumal2020introduction,zhang2016first}. A differentiable function $F:\M \xrightarrow{} \sR$ is \textit{geodesic} $L_0$-\textit{Lipschitz} if for any $w \in \M$, $\| \grad F(w)\|_w \leq L_0$. The function $F$ is called \textit{geodesic} $L_1$-\textit{smooth} if for any $w \in \M$ and $w' = {\rm Exp}_w(\xi)$, we have $|F(w') - F(w) - \langle \grad F(w) , \xi \rangle_w | \leq  \frac{L_1}{2}\| \xi \|_w^2$. 

Geodesic convexity \cite{zhang2016first} is an extension of convexity in the Euclidean space. A set $\gW \subseteq \M$ is \textit{geodesic convex} if for any two points in the set, there exists a geodesic in the set joining them. A function $F: \gW \xrightarrow{} \sR$ is called \textit{geodesic convex} if for any $w, w' \in \gW$, it satisfies $F(\gamma(t)) \leq (1-t) F(w) + t F(w')$ for all $t \in [0,1]$, where $\gamma$ is the geodesic such that $\gamma(0) = w, \gamma(1) = w'$. If the function $F$ is differentiable, an equivalent characterization of geodesic convexity is $F(w') \geq F(w) + \langle \grad F(w), \xi \rangle_w$ for any $w \in \gW$, $w' = {\rm Exp}_w(\xi)$. In addition, a function $F$ is called \textit{geodesic $\beta$-strongly convex} if for any $w, w'  = {\rm Exp}_w(\xi) \in \gW$, it satisfies $F(w') \geq F(w) + \langle \grad F(w), \xi \rangle_w + \frac{\beta}{2}d^2(w, w')$ for some $\beta > 0$.  

Further, we introduce Riemannian Polyak–\L{}ojasiewicz (PL) condition \cite{zhang2016riemannian,kasai2018riemannian,han2021improved}, which is weaker than the geodesic strong convexity. A function $F:\M \xrightarrow{} \sR$ is said to satisfy the \textit{Riemannian PL condition} if for any $w \in \M$, there exists $\tau > 0$ such that $F(w) - F(w^*) \leq \tau \| \grad F(w) \|^2_w$ where $w^*$ is a global minimizer of $F$ on $\M$.

Finally, we recall a trigonometric distance bound for Riemannian manifolds with lower bounded sectional curvature, which is crucial in convergence analysis for geodesic convex optimization.

\begin{lemma}[Trigonometric distance bound \cite{bonnabel2013stochastic,zhang2016first,zhang2016riemannian,han2021riemannian}]
\label{curvature_lemma}
Let $w_0, w_1, w_2 \in \gW \subseteq \M$ lie in a totally normal neighbourhood of a Riemannian manifold with curvature lower bounded by $\kappa_{\rm min}$, and $l_0 = {\rm dist}(w_0, w_1), l_1 = {\rm dist}(w_1, w_2)$ and $l_2 = {\rm dist}(w_0, w_2)$. Denote $\theta$ as the angle on $T_{w_0}\M$ such that $\cos(\theta) =  \frac{1}{l_0 l_2}\langle {\rm Exp}_{w_0}^{-1}(w_1), {\rm Exp}_{w_0}^{-1}(w_2) \rangle_{w_0}$. Let $D_\gW$ be the diameter of $\gW$, i.e., $D_\gW := \max_{w, w' \in \gW} {\rm dist}(w, w')$. Define the curvature constant $\varsigma = \frac{\sqrt{|\kappa_{\rm min}|}D_\gW}{\tanh(|\sqrt{|\kappa_{\rm min}|} D_\gW)}$ if $\kappa_{\rm min} < 0$ and $\varsigma = 1$ if $\kappa_{\rm min} \geq 0$. Then, we have $l_1^2 \leq  \varsigma l_0^2 + l_2^2 - 2 l_0 l_2 \cos(\theta)$. 
\end{lemma}

\label{dp_review_sect}

\paragraph{Differential privacy.} Let $D = \{z_1, \ldots, z_n\} \subset \gD^n$ be a dataset. A neighbouring dataset of $D$, denoted as $D'$ is a dataset that differs in only one sample from $D$. The neighbouring relation is denoted as $D \sim D'$. We first recall the definition of $(\epsilon, \delta)$-differential privacy (DP) \cite{dwork2014algorithmic}, which is defined on arbitrary measurable space $\M$ (not necessarily a Riemannian manifold). 

\begin{definition}[$(\epsilon, \delta)$-Differential privacy]
\label{differential_privacy_def}
A randomized mechanism $\gR: \gD^n \xrightarrow{} \M$ is called $(\epsilon, \delta)$-differentially private on $\M$ if for any neighbouring datasets $D, D' \subset \gD^n$ and any measurable space $\gA \subseteq \M$, we have $\sP (\gR(D) \in \gA) \leq \exp(\epsilon) \, \sP(\gR(D') \in \gA) + \delta$.
\end{definition}

In addition, we make use of R\'enyi differential privacy (RDP) \cite{mironov2017renyi}, which enjoys a tighter privacy bound under composition and subsampling \cite{wang2019subsampled}. Given neighbouring datasets $D, D'$, we first define the cumulant generating function of a mechanism $\gR$ as
\begin{equation*}
    K_{\gR, (D, D')}(\lambda) := \log \sE_{o \sim \gR(D)} [\exp(\lambda \, \gL_{\gR(D) \| \gR(D')}) ] = \log \sE_{o \sim \gR(D)} \Big[ \Big( \frac{p(\gR(D) = o)}{p(\gR(D') = o)} \Big)^{\lambda}  \Big], 
\end{equation*}
where $\gL_{\gR(D) \| \gR(D')} = \log \big(\frac{p(\gR(D) = o)}{p(\gR(D')  = o)} \big)$ is known as the privacy loss random variable at $o \sim \gR(D)$ \cite{dwork2010boosting,dwork2014algorithmic}. When maximized over all the neighbouring datasets, $K_{\gR}(\lambda) := \sup_{D \sim D'} K_{\gR,(D, D')}(\lambda)$ is called the $\lambda$-th moment of the mechanism \cite{abadi2016deep}.

\begin{definition}[$(\alpha,\rho)$-R\'enyi differential privacy \cite{mironov2017renyi}]
For $\alpha \geq 1$ and $\rho > 0$, a randomized mechanism $\gR: \gD^n \xrightarrow{} \M$ is called $(\alpha, \rho)$-R\'enyi differentially private if $ \frac{1}{\alpha - 1} K_{\gR}(\alpha-1) \leq \rho$. 
\end{definition}


\begin{proposition}[Relationship between RDP and $(\epsilon, \delta)$-DP \cite{mironov2017renyi}]
\label{prop_rdp_dp}
If a mechanism $\gR$ satisfies $(\alpha, \rho)$-R\'enyi differential privacy, then it satisfies $(\rho + \log(1/\delta)/(\alpha - 1), \delta)$-differential privacy.
\end{proposition}

The notions of differential privacy introduced above are well-defined on Riemannian manifolds, which is a measurable space under the Borel sigma algebra \cite{pennec2006intrinsic}. However, a systematic approach for preserving differential privacy when the parameters of interest are on Riemannian manifolds has not been studied. A recent work~\cite{reimherr2021differential} proposes differentially private Fr\'echet mean computation over general Riemannian manifolds by output perturbation. Nevertheless, computing the Fr\'echet mean is a special problem instance of \eqref{main_problem}, which we can solve via our proposed general framework under differential privacy. See more detailed discussions and comparisons in Section \ref{application_sect}. 




\section{Differential privacy on Riemannian manifolds}
\label{diff_privacy_manifold_sect}

This section proposes many tools for preserving and analyzing differential privacy on Riemannian manifolds. Proofs for the results in this section are deferred to the supplementary. 

First, we generalize the Gaussian mechanism \cite{dwork2014algorithmic} from the Euclidean space to Riemannian manifolds. One approach is to directly add noise on the manifold following an intrinsic Gaussian distribution \cite{pennec2006intrinsic}. However, this strategy faces two challenges. First, it is required to bound the sensitivity in terms of the Riemannian distance, which could be difficult particularly for negatively curved manifolds, such as the symmetric positive definite manifold. Second, the generalization suffers from metric distortion by curvature and it requires a nontrivial adaptation of the proof strategy in the Euclidean space \cite{dwork2014algorithmic}. Nevertheless, it is worth mentioning that the Laplace mechanism \cite{dwork2006calibrating} can be generalized with the triangle inequality of the Riemannian distance as has been done recently in \cite{reimherr2021differential}.

Instead, we consider directly adding noise to the tangent space of the manifold following an isotropic Gaussian distribution with respect to the Riemannian metric. In this case, we can measure the sensitivity on the tangent space. We highlight that although the tangent space can be identified as a Euclidean space, the proposed strategy differs from the classic Gaussian mechanism, which adds isotropic noise to each coordinate in the Euclidean space. 

To this end, we define the tangent space Gaussian distribution as follows.

\begin{definition}[Tangent space Gaussian distribution]
\label{gaussian_tangent_def}
For any $w \in \M$, a tangent vector $\xi \in T_w\M$ follows a tangent space Gaussian distribution at $w$, denoted as $\xi \sim \gN_{w}(\mu, \sigma^2)$ with mean $\mu \in T_w\M$ and standard deviation $\sigma > 0$ if its density is given by $p_w(\xi) = C^{-1}_{w, \sigma} \exp(- \frac{\| \xi - \mu \|^2_w}{2\sigma^2})$, where $C_{w, \sigma}$ is the normalizing constant.
\end{definition}

\begin{remark}
\label{rmk_Gaussian}
In a normal coordinate system of the tangent space, we remark that $\gN_w(\mu, \sigma^2)$ is equivalent to a standard multivariate Gaussian with covariance as a function of the metric tensor of the tangent space. Denote $\vec{\xi} \in \sR^{d}$ as the vectorization of the tangent vector $\xi \in T_w\M \simeq \sR^d$ in the normal coordinates. The density can then be written as $p_w(\xi) = C^{-1}_{w,\sigma} \exp(- \frac{1}{{2\sigma^2}} (\vec{\xi} - \vec{\mu})^\top G_w ( \vec{\xi} - \vec{\mu}))$, where $G_w$ is the (symmetric positive definite) metric tensor at $w$. This is a standard Gaussian distribution with mean $\vec{\mu}$ and covariance $\sigma^2 G_w^{-1}$, i.e., $\vec{\xi} \sim \gN(\vec{\mu}, \sigma^2 G_w^{-1})$. The normalizing constant is given by $C_{w, \sigma} = \sqrt{(2\pi)^d \sigma^2 |G_w^{-1}|}$. 
\end{remark}

Next, we introduce a generalization of the Gaussian mechanism on the tangent space. We stress that the following Gaussian mechanism depends on the sensitivity measured in the Riemannian metric.

\begin{proposition}[Tangent space Gaussian mechanism]
\label{gaussian_mech_manifold}
Given a query function $H: \gD^n \xrightarrow{} T_w\M$ on tangent space at $w \in \M$, let $\Delta_H := \sup_{D, D' \subset \gD^n: D \sim D'} \| H(D) - H(D') \|_w$ be the global sensitivity of $H$ with respect to the Riemannian metric. Define $\gR(D) = H(D) + \xi$ where $\xi \sim \gN_w(0, \sigma^2)$ with $\sigma^2 \geq {2\log(1.25/\delta)} \Delta_H^2/\epsilon^2$. Then, $\gR$ is $(\epsilon, \delta)$-differentially private. 
\end{proposition}

To show privacy guarantees in the subsequent sections, we adapt the moments accountant technique in the Euclidean space \cite{abadi2016deep} to Riemannian manifolds, which results in a tighter bound compared to the advanced composition \cite{dwork2014algorithmic}. To achieve this, we first provide lemmas that bound the moments of a tangent space Gaussian mechanism (Proposition \ref{gaussian_mech_manifold}). The proof strategy is motivated by the connection of the $(\epsilon, \delta)$-differential privacy (Definition \ref{differential_privacy_def}) and R\'enyi differential privacy \cite{mironov2017renyi} established in \cite{wang2019subsampled}.

The next two lemmas show upper bounds on $K_\gR(\lambda)$ under full datasets (Lemma \ref{moments_bound_full}) as well as under subsampling (Lemma \ref{moments_bound_subsample}).

\begin{lemma}[Moments bound]
\label{moments_bound_full}
Consider a query function $H : \gD^n \xrightarrow{} T_w\M$ for some $w \in \M$. Given a dataset $D = \{z_1, \ldots, z_n\} \subset \gD^n$ and suppose $H(D) = \frac{1}{n} \sum_{i=1}^n h(z_i)$ and $h$ is geodesic $L_0$-Lipschitz. Let $\gR(D) = H(D) + \xi$, where $\xi \sim \gN_w(0, \sigma^2)$. Then, the $\lambda$-th moment of $\gR$ satisfies $K_\gR(\lambda) \leq \frac{2 \lambda (\lambda + 1) L_0^2}{n^2 \sigma^2}$.
\end{lemma}

\begin{lemma}[Moments bound under subsampling]
\label{moments_bound_subsample}
Under the same settings as in Lemma \ref{moments_bound_full}, consider $D_{\rm sub}$ to be a subset of size $b$ where samples are selected from $D$ without replacement. Let $\gR(D) = \frac{1}{b} \sum_{z \in D_{\rm sub}} h(z) + \xi$, where $\xi \sim \gN_w(0, \sigma^2)$. Suppose $\sigma^2 \geq \frac{4L_0^2}{b^2}, b < n, \lambda \leq 2 \sigma^2 \log(n/(b(\lambda + 1)(1 + \frac{b^2 \sigma^2}{4 L_0^2})))/3$. Then $\lambda$-th moment of $\gR$ satisfies $K_\gR(\lambda) \leq \frac{15  (\lambda+1) L_0^2}{n^2 \sigma^2}$.
\end{lemma}

It is worth highlighting that the bound given in Lemma \ref{moments_bound_subsample} asymptotically matches the bound derived from the moments accountant \cite{abadi2016deep} when $b/n \xrightarrow{} 0$. In addition, we observe that under small $\lambda$ and large $\sigma$, subsampling does not improve the bound compared to using the full dataset.

\section{Differentially private Riemannian (stochastic) gradient descent}
\label{diff_private_riemannian_opt_sect}

In this section, we introduce differentially private Riemannian (stochastic) gradient descent (Algorithm \ref{DPRGD}), where we add noise following the tangent space Gaussian distribution $\gN_w(0, \sigma^2)$. We show under proper choice of parameters, the algorithm preserves both the privacy guarantee as well as utility guarantees under various function classes on Riemannian manifolds. Proofs for this section are included in the supplementary.


In particular, in Algorithm \ref{DPRGD}, the samples are selected without replacement following \cite{wang2019subsampled}, and thus, when $b = n$, we recover the full gradient descent. The noise variance is chosen as $\sigma^2 = c \frac{T\log(1/\delta) L_0^2}{n^2\epsilon^2}$ to ensure $(\epsilon, \delta)$-differential privacy (Theorem \ref{privacy_guarantee_theorem}) for some constant $c$.  We remark that the choice of $\sigma$ matches the standard results in the Euclidean space for gradient descent \cite{wang2017differentially} and for stochastic gradient descent \cite{abadi2016deep,bassily2019private,wang2019efficient} up to some constants that may depend on the manifold of interest. The output of the the algorithm depends on the function class of the objective, discussed in Section \ref{convergence_sect}.

\begin{algorithm}[!t]
 \caption{Differentially private Riemannian (stochastic) gradient descent}
 \label{DPRGD}
 \begin{algorithmic}[1]
  \STATE \textbf{Input:} Data $D = {z_1,\ldots, z_n}$, initialization $w_0 \in \M$, Lipschitz constant $L_0$ for $f(w;z)$, privacy parameters $\epsilon, \delta$, batch size $b$, stepsize $\eta_t$, constant $c > 0$.
  \FOR{$t = 0,\ldots,T-1$}
  \STATE Select $\gB_t \subseteq D$ of size $b$ where samples are randomly selected uniformly without replacement.
  \STATE Set $\sigma^2 = \frac{c T\log(1/\delta) L_0^2}{n^2 \epsilon^2} \geq \frac{4L_0^2}{b^2}$ and compute $\zeta_t = \frac{1}{b} \sum_{z \in \gB_t} \grad f(w_t; z) + \epsilon_t$, where $\epsilon_t \sim \gN_w(0, \sigma^2)$. 
  \STATE Update $w_{t+1} = {\rm Exp}_{w_t} \big( - \eta_t \zeta_t \big)$.
  \ENDFOR
  \STATE \textbf{Output 1:} $w^{\rm priv} = w_T$.
  \STATE \textbf{Output 2:} $w^{\rm priv}$ as uniformly selected at random from $\{w_0,\ldots, w_{T-1} \}$.
  \STATE \textbf{Output 3:} $w^{\rm priv}$ is the geodesic averaging of past iterates detailed in Theorem \ref{g_convex_theorem}, \ref{g_strongly_convex_theorem}.
 \end{algorithmic} 
\end{algorithm}

\subsection{Privacy guarantees}

\begin{theorem}[Privacy guarantee]
\label{privacy_guarantee_theorem}
Algorithm \ref{DPRGD} is $(\epsilon, \delta)$-differentially private.
\end{theorem}
\begin{proof}
The idea is to bound the moment of the randomized mapping $K_{\gR_t}(\lambda)$ every iteration using Lemma \ref{moments_bound_full} for gradient descent and Lemma \ref{moments_bound_subsample} for stochastic gradient descent. Then by composability theorem \cite[Theorem 2.1]{abadi2016deep} as well as the connection between RDP to DP in Proposition \ref{prop_rdp_dp}, we can ensure the differential privacy. Detailed proof can be found in the supplementary.


\end{proof}

\subsection{Convergence guarantees}
\label{convergence_sect}

For convergence analysis, we start by making an assumption that all the iterates $w_0, \ldots, w_T$ stay bounded within a compact support $\gW \subseteq \M$ that contains a stationary point $w^*$ (i.e., $\grad F(w^*) = 0$). Let $c_l > 0$ such that $G_w \succeq c_l I_d$ satisfies for all $w \in \gW$. We first show that the expected norm of the noise injected gradient can be bounded as follows.
\begin{lemma}
\label{lemma_bound}
Suppose $f(w; z)$ is geodesic $L_0$-Lipschitz for any $z$. Consider a batch $\gB$ of size $b$ and $w \in \gW$. Let $\zeta = \frac{1}{b} \sum_{z \in \gB} \grad f(w; z) + \epsilon$ where $\epsilon \sim \gN_w(0, \sigma^2)$. Then, we have $\sE[\zeta] = \grad F(w)$ and $\sE \| \zeta  \|_w^2 \leq  L_0^2 + d\sigma^2 c_l^{-1}$, where the expectation is over randomness in both $\gB$ and $\epsilon$.
\end{lemma}

Next we show utility guarantees under various function classes on Riemannian manifolds, including geodesic (strongly) convex, general nonconvex and functions that satisfy Riemannian PL condition. It has been shown that many nonconvex problems in the Euclidean space are in fact geodesic (strongly) convex or satisfy the Riemannian PL condition on the manifold. This allows tighter utility bounds compared to differentially private projected gradient methods. Some examples are given in Section~\ref{application_sect}.

\paragraph{Geodesic convex optimization.}
When $f(w; z)$ is geodesic convex over $\gW$ for any $z \in \gD$, the stationary points $w^*$ is a global minimum of $F(w)$. The utility of Algorithm \ref{DPRGD} is measured as the expected empirical excess risk $\sE[F(w^{\rm priv})] - F(w^*)$.

\begin{theorem}[Utility under geodesic convex optimization]
\label{g_convex_theorem}
Suppose $f(w; z)$ is geodesic convex, geodesic $L_0$-Lipschitz over $\gW$. Assume $\gW$ to be a totally normal neighbourhood with diameter $D_\gW$. Let $\varsigma$ be the curvature constant of $\gW$ defined in Lemma \ref{curvature_lemma}. Consider Algorithm \ref{DPRGD} with \textbf{output 3} where $w^{\rm priv} = \bar{w}_{T-1}$ is computed by geodesic averaging as follows: set $\bar{w}_1 = w_1$ and $\bar{x}_{t+1} = {\rm Exp}_{\bar{w}_t}\big( \frac{1}{t+1} {\rm Exp}^{-1}_{\bar{x}_t}(x_{t+1}) \big)$. Set $\eta_t = \frac{D_\gW}{\sqrt{(L_0^2 + d \sigma^2 c_l^{-1}) \varsigma T}}$. Then $w^{\rm priv}$ satisfies
\begin{equation*}
    \sE[F(w^{\rm priv})] - F(w^*) \leq D_\gW \sqrt{\frac{(L_0^2 + d \sigma^2 c_l^{-1}) \varsigma}{T}} = O\Big( \frac{\sqrt{d \log(1/\delta) c_l^{-1} \varsigma} L_0 D_\gW }{n\epsilon} \Big),
\end{equation*}
for the choice of $\sigma^2$ in Algorithm \ref{DPRGD} and $T = n^2$.
\end{theorem}

First, we see that Theorem \ref{g_convex_theorem} indicates the same utility bound under both gradient descent and stochastic gradient descent. This matches the results in the Euclidean space for convex functions \cite{bassily2014private} up to a square root factor of the curvature constant $\varsigma$. 
In addition, we highlight that the bound also depends on the intrinsic dimension of the manifold, rather than the ambient space.

\paragraph{Geodesic strongly convex optimization.}
Under geodesic strong convexity, the global optimizer $w^*$ is unique and we use the same measure for bounding the utility.

\begin{theorem}[Utility under geodesic strongly convex optimization]
\label{g_strongly_convex_theorem}
Suppose $f(w; z)$ is geodesic $\beta$-strongly convex, geodesic $L_0$-Lipschitz over $\gW$. Assume $\gW$ to be a totally normal neighbourhood with the curvature constant $\varsigma$. Consider Algorithm \ref{DPRGD} with \textbf{output 3} where $w^{\rm priv} = \bar{w}_{T-1}$ is the geodesic averaging by setting $\bar{w}_1 = w_1$, $\bar{w}_{t+1} = {\rm Exp}_{\bar{w}_t}\big( \frac{2}{t+1} {\rm Exp}^{-1}_{\bar{w}_t} (w_{t+1}) \big)$. Set $\eta_t = \frac{1}{\beta (t+1)}$. Then $w^{\rm priv}$ satisfies
\begin{equation*}
    \sE[F(w^{\rm priv})] - F(w^*) \leq  \frac{2 \varsigma (L_0^2 + d \sigma^2 c_l^{-1})}{\beta T} = O \Big( \frac{\beta^{-1} d \log(1/\delta) c_l^{-1} \varsigma  L_0^2}{ n^2 \epsilon^2} \Big).
\end{equation*}
with the choice of $\sigma^2$ and $T = n^2$.
\end{theorem}

\paragraph{Optimization under Riemannian PL condition.}
Next, we consider the case when the objective $F$ satisfies the Riemannian Polyak–\L{}ojasiewicz (PL) condition. It is known that Polyak-\L{}ojasiewicz is a sufficient condition to establish linear convergence to global minimum \cite{karimi2016linear}. In addition, the Riemannian PL condition includes the geodesic strongly convexity as a special case \cite{zhang2016riemannian}. 

\begin{theorem}[Utility under Riemannian Polyak–\L{}ojasiewicz condition]
\label{PL_condition_utility_theorem}
Suppose $f(w; z)$ is geodesic $L_0$-Lipschitz and $L_1$-smooth over $\gW$ and $F(w)$ satisfies the Riemannian PL condition with parameter $\tau$, i.e., $F(w) - F(w^*) \leq \tau \| \grad F(w) \|^2_w$. Consider Algorithm \ref{DPRGD} with \textbf{output 1}. Set $\eta_t = \eta < \min\{1/L_1, 1/\tau \}$ and $T = \log(\frac{n^2\epsilon^2 c_l}{dL_0^2 \log(1/\delta)})$. Then, $w^{\rm priv}$ satisfies 
\begin{equation*}
    \sE[F(w^{\rm priv})] - F(w^*) \leq O \Big( \frac{\tau^{-1} d \log(n) \log(1/\delta) c_l^{-1} L_0^2 \Delta_0}{n^2\epsilon^2} \Big),
\end{equation*}
where we denote $\Delta_0 := F(w_0) - F(w^*)$. 
\end{theorem}

\paragraph{General nonconvex optimization.}
Under general nonconvex objectives, we show utility bound with respect to the expected gradient norm squared, i.e., $\sE \| \grad F(w^{\rm priv}) \|^2_{w^{\rm priv}}$, which has also been considered in the Euclidean space \cite{zhang2017efficient,wang2017differentially}.

\begin{theorem}[Utility under general nonconvex optimization]
\label{general_nonconvex_theorem}
Suppose $f(x;w)$ is geodesic $L_0$-Lipschitz and $L_1$-smooth, possibly nonconvex. Consider Algorithm \ref{DPRGD} with \textbf{output 2}. Set $\eta_t = \eta = 1/L_1$ and $T = \frac{\sqrt{L_1} n \epsilon}{\sqrt{d \log(1/\delta) c_l^{-1}}L_0}$. Then $w^{\rm priv}$ satisifes
\begin{equation*}
    \sE \| F(w^{\rm priv}) \|^2_{w^{\rm priv}} \leq O \Big( \frac{\sqrt{d L_1 \log(1/\delta) c_l^{-1}} L_0}{n\epsilon}   \Big).
\end{equation*}
\end{theorem}

\begin{remark}[Extending utility guarantees to other Riemannnian optimization methods]
In this section, we have shown utility guarantees for the vanilla Riemannian gradient descent and stochastic gradient descent under differential privacy. We remark that such a strategy can be applied to more advanced solvers while preserving the same privacy budget. One example is to use line search methods to select the stepsize. Under current parameter settings, the same privacy guarantees can be preserved. In addition, provided that the update direction $\zeta_t$ is `close' to the full gradient $\grad F(w_t)$ (as in \cite{boumal2019global}), the same utility bounds also hold. 
\end{remark}

\section{Applications}
\label{application_sect}

Here we explore two applications to demonstrate the efficacy of the proposed framework of differentially private Riemannian optimization. The experiments are implemented in Matlab using ManOpt package \cite{boumal2014manopt} on an i7-8750H CPU. The codes can be found on \url{https://github.com/andyjm3}.

\subsection{Principal eigenvector computation over sphere manifold}
We first consider the problem of computing leading eigenvector of a sample covariance matrix \cite{shamir2015stochastic,zhang2016riemannian} as 
\begin{equation}
    \min_{w \in \gS^{d}} \Big\{ F(w) =  \frac{1}{n} \sum_{i =1}^n f(w; z_i) = - \frac{1}{n} \sum_{i=1}^n w^\top \Big( z_i z_i^\top \Big) w \Big\} \label{sphere_problem}
\end{equation}
where $\gS^{d} = \{ w \in \sR^{d+1} : \|w  \|_2 = w^\top w = 1 \}$ is the sphere manifold of intrinsic dimension $d$ and $z_1, \ldots, z_n \in \sR^{d+1}$ are zero-centered samples.  Sphere manifold is an embedded submanifold of $\sR^{d+1}$ with the tangent space given as $T_w\gS^{d} = \{ \xi \in \sR^{d+1}: w^\top \xi = 0 \}$. A common Riemannian metric used is the Euclidean metric, i.e., $\langle \xi, \zeta \rangle_w =  \xi^\top \zeta $ for any $\xi, \zeta \in T_w\gS^{d}$. The exponential map is derived as ${\rm Exp}_{w}(\xi) = \cos(\| \xi \|_2) w + \sin(\| \xi \|_2) \xi/\|\xi \|_2$. Let the Euclidean gradient be denoted as $\nabla F(w)$. The Riemannian gradient on $T_w\gS^{d}$ is $\grad F(w) = (I_{d+1} - ww^\top) \nabla F(w)$, where $I_{d+1}$ is the identity matrix of size $(d+1)\times (d+1)$.

The next theorem shows that problem \eqref{sphere_problem}, although being nonconvex in the ambient Euclidean space, satisfies the Riemannian PL condition locally around the optimality.

\begin{theorem}[\cite{zhang2016riemannian}]
\label{pl_pca_theorem}
Let $A = \frac{1}{n} \sum_{i=1}^n z_iz_i^\top$ and denote $\lambda_{i}$ as the $i$-th largest eigenvalues of $A$. Assume $\nu = \lambda_1 - \lambda_2 > 0$. Then, the problem \eqref{sphere_problem} locally satisfies the Riemannian PL condition with parameter $\tau = O(d/(q^2 \nu))$ with probability $1-q$.
\end{theorem}

\paragraph{Utility.} 
First we see that the metric tensor $G_w = I_{d}$ and hence $c_l = 1$. This is because given an orthonormal basis on the tangent space $T_w\gS^{d}$, denoted as $B \in \sR^{(d+1) \times d}$, we have $\langle \xi, \zeta \rangle_w = \langle {\xi}, \zeta \rangle_2 = \langle B \vec{\xi}, B \vec{\zeta} \rangle_2 = \langle \vec{\xi}, \vec{\zeta} \rangle_2$, where $\vec{\xi}, \vec{\zeta} \in \sR^{d}$ is the vectorization under the coordinate transformation given by $B$. In addition, the geodesic Lipschitz constant $L_0$ is bounded as $\| \grad f(w; z_i) \|_w = \| 2(I_{d+1} - ww^\top) (z_i z_i^\top) w \|_2 \leq 2\|(z_i z_i^\top) w \|_2 \leq \vartheta := \max_{i} z_i^\top z_i$. Then applying Theorem \ref{PL_condition_utility_theorem} for utility under Riemannian PL condition, we can show if properly initialized, the expected empirical excess risk of Algorithm \ref{DPRGD} is bounded as $\sE[F(w^{\rm priv})] - F(w^*) \leq  \widetilde{O} \big( \frac{ q^2 \nu \vartheta^2 \Delta_0}{n^2\epsilon^2} \big)$, where we have $c_l = 1, L_0 = \vartheta$ and we use $\widetilde{O}$ to hide logarithmic factors. 

\begin{remark}[Comparison with perturbed projected gradient descent]
Considering the alternative approach using projected gradient descent with gradient perturbation in the ambient space for solving the problem (\ref{sphere_problem}), one can only guarantee a utility of $\sE[\| \nabla F(w^{\rm priv})]  \|_2^2] \leq \widetilde{O} \big( \frac{ \vartheta}{n\epsilon} \big)$ due to the nonconvexity of the problem \cite{wang2017differentially}. This results in a looser bound in $n$ compared to our obtained utility guarantee above.
\end{remark}

\paragraph{Sampling from the tangent space Gaussian distribution.}
Based on the argument above, the tangent space Gaussian distribution $\gN_w(0, \sigma^2)$, from Definition \ref{gaussian_tangent_def}, reduces to the classic isotropic Gaussian distribution $\gN(0, \sigma^2 I_{d})$ in $\sR^{d}$. Hence, sampling from $\gN_w(0, \sigma^2)$ can be achieved by first sampling from $\gN(0, \sigma^2 I_{d})$ and then transforming using a basis matrix $B$.

\paragraph{Experiment settings and results.}
We follow the same procedures as in \cite{shamir2015stochastic} to generate the sample matrix $Z = [z_1^\top; \ldots; z_n^\top ] \in \sR^{n \times (d+1)}$. Specifically, we construct a $(d+1)\times (d+1)$ diagonal matrix $\Sigma = {\rm diag}(1, 1-1.1\nu, \ldots,1-1.4\nu, |x_1|/(d+1), |x_2|/(d+1), \ldots)$ where $\nu$ is the eigengap defined in Theorem \ref{pl_pca_theorem} and $x_1, x_2, \ldots$ are standard Gaussian random variables. Then construct $Z = U \Sigma V$ where $U, V$ are $n \times (d+1)$ and $(d+1) \times (d+1)$ are random column orthonormal matrices. Thus $Z$ has the same spectrum as $\Sigma$. We generate noise following Algorithm \ref{DPRGD} where the parameters $T = \log(n^2\epsilon^2/((d+1)L_0^2\log(1/\delta)))$ and $\sigma^2 = T \log(1/\delta)L_0^2/(n^2\epsilon^2)$ according to Theorem \ref{PL_condition_utility_theorem} for Riemannian PL condition. We set $\epsilon = 0.1, \delta = 10^{-3}, \nu = 10^{-3}, d+1 = 50$ and $L_0$ is estimated from the samples. We compare Algorithm \ref{DPRGD} with full gradient $(b = n)$, denoted as DP-RGD against the projected gradient descent (with noise added in the ambient space), denoted as DP-PGD. In fact, the projected gradient descent on the sphere approximates the Riemannian gradient descent because the former updates by $w_{t+1} = \frac{w_t + \zeta_t}{\|w_t + \zeta_t \|_2}$ where $\zeta_t = \nabla F(w_t) + \epsilon_t$. This approximates the exponential map to the first order and is known as the retraction \cite{absil2012projection,boumal2020introduction}. For this reason and the purpose of comparability, we set the same noise variance and the max iterations for both DP-RGD and DP-PGD. We show the expected empirical excess risk under different sample size in Figure \ref{pca_fig}, with the stepsize tuned and fixed to be $\eta = 0.2$ for both algorithms and results averaged over 20 runs with different initializations. From the figure, we see an improved utility of proposed DP-RGD, demonstrating the benefit of using intrinsic Riemannian update.

\subsection{Fr\'echet mean computation over symmetric positive definite manifold}

The second application we consider is Fr\'echet mean computation \cite{bhatia2009positive,jeuris2012survey} over the manifold of $r \times r$ symmetric positive definite (SPD) matrices, denoted as $\sS_{++}^r$. Specifically, given a set of SPD matrices $X_1, \ldots, X_n \in \sS_{++}^r$, the goal is to find a center $W \in \sS_{++}^r$ by minimizing an empirical average of the squared Riemannian distance to the samples:
\begin{equation}
    \min_{W \in \sS_{++}^r} \Big\{ F(W) = \frac{1}{n} \sum_{i=1}^n f(W; X_i) = \frac{1}{n} \sum_{i=1}^n \| {\rm logm}( W^{-1/2} X_i W^{-1/2}) \|_{\rm F}^2  \Big\}, \label{spd_problem}
\end{equation}
where ${\rm logm}(\cdot)$ represents the principal matrix logarithm. The tangent space of the set of SPD matrices is given as $T_W\sS_{++}^r = \{ U \in \sR^{r \times r}: U^\top = U \}$, i.e., the set of symmetric matrices. It can be shown the function $f(W; X_i)$ in the problem \eqref{spd_problem} is the Riemannian distance squared ${\rm dist}^2(W, X_i)$ associated with the affine-invariant Riemannian metric \cite{bhatia2009positive}, defined as $\langle U, V \rangle_W = \trace(U W^{-1} V W^{-1})$. The exponential map is derived as ${\rm Exp}_W(U) = W^{1/2} {\rm expm}(W^{-1/2} U W^{-1/2}) W^{1/2}$, where ${\rm expm}(\cdot)$ denotes the principal matrix exponential. The Riemannian gradient is computed by $\grad F(W) = W \nabla F(W) W$. Similarly, the problem \eqref{spd_problem} is known to be nonconvex in the Euclidean space while it is geodesic strongly convex on the SPD manifold \cite{zhang2016riemannian,zhang2016first}.

\begin{theorem}[\cite{zhang2016riemannian,zhang2016first}]
\label{spd_problem_strongly_convex_theorem}
Problem \eqref{spd_problem} is $2$-geodesic strongly convex. 
\end{theorem}

\paragraph{Utility.} 
To show the utility guarantees, we first show that the Lipschitz constant can be bounded. Specifically, the Riemannian gradient of problem \eqref{spd_problem} can be derived as $\grad f(W; X_i) = W {\rm logm}( W^{-1} X_i )$ and $\| \grad f(W; X_i) \|_W = \| {\rm logm}( W^{-1} X_i )  \|_F = {\rm dist}(W, X_i) \leq D_\gW = L_0$. Also, it is known the SPD manifold with affine-invariant metric is negatively curved, then we have the curvature constant $\zeta \leq 2 \sqrt{|\kappa_{\min}|} D_\gW +1$ according to Lemma \ref{curvature_lemma}. Then applying Theorem \ref{g_strongly_convex_theorem}, we show the utility of Algorithm \ref{DPRGD} is $\sE[F(W^{\rm priv})] - F(W^*) \leq  \widetilde{O} \big( \frac{d c_l^{-1} \sqrt{|\kappa_{\rm min}|} D_\gW^3}{ n^2 \epsilon^2} \big)$, where the intrinsic dimension of the SPD manifold $\sS_{++}^r$ is $d = r(r+1)/2$.

\begin{remark}[Comparison to utility obtained in \cite{reimherr2021differential}]
Here we compare the utility of proposed Algorithm to the result in \cite{reimherr2021differential}. First we highlight that in \cite{reimherr2021differential}, by output perturbation with Laplace noise, $W^{\rm priv}$ is shown to satisfy $\epsilon$-\textit{pure differential privacy} (with $\delta = 0$). The utility is given by ${\rm dist}^2(W^{\rm priv}, W^*) \leq O\big( \frac{d^2 D_\gW^2}{n^2 \epsilon^2} \big)$.
In contrast, Algorithm \ref{DPRGD} preserves $(\epsilon, \delta)$-differential privacy with an utility of ${\rm dist}^2(W^{\rm priv}, W^*) \leq \widetilde{O}\big( \frac{d D_\gW^3}{n^2\epsilon^2} \big)$ by applying the triangle inequality to the expected empirical excess risk and ignoring other factors. 
\end{remark}

\paragraph{Sampling from the tangent space Gaussian distribution.}
Sampling $\epsilon \sim \gN_w(0, \sigma^2) \propto \exp( - \frac{1}{2\sigma^2} \| \epsilon \|^2_W )$ can be performed via standard implementation of the random walk Metropolis-Hastings \cite{robert1999monte}. Particularly, we can repeatedly sample $\vec{\epsilon'}$ from a proposal Gaussian distribution on $\sR^{r(r+1)/2}$ conditional on the previous iterate and evaluate $\exp( - \frac{1}{2\sigma^2} \| \epsilon' \|^2_W )$.

\paragraph{Experiment settings and results.}
We follow the steps in \cite{reimherr2021differential} to generate synthetic samples on SPD manifold $\sS_{++}^r$ following the Wishart distribution $W(I_{r}/r, r)$ with a diameter bound $D_\gW$. The optimal solution $W^*$ is obtained by running Riemannian gradient descent (RGD) on problem \eqref{spd_problem} until the gradient norm falls below $10^{-14}$. For the example, we choose $r = 2$ and set $T = n$, which we find empirically better than $n^2$ (as suggested in Theorem \ref{g_strongly_convex_theorem}). We choose $\sigma^2 = T \log(1/\delta) D_w^2/(n^2\epsilon^2)$ where $\epsilon = 0.1, \delta = 10^{-3}, D_\gW = 1$, $\eta = 0.01$. We compare the proposed DP-RGD with the differentially private Fr\'echet mean (DP-FM) by output perturbation in \cite{reimherr2021differential}. Following the procedures in \cite{reimherr2021differential}, we first obtain a non-privatized Fr\'echet mean $\widehat{W}$ by running RGD. Then we sample from an intrinsic Laplace distribution on $\sS_{++}^r$ (by the steps in \cite{hajri2016riemannian}) with footprint $\widehat{W}$ and $\sigma = \Delta_{\rm FM}/\epsilon$ where $\Delta_{\rm FM} = 2D_\gW/n = 2/n$ is the sensitivity of the Fr\'echet mean on SPD manifold (\cite[Theorem 2]{reimherr2021differential}). We plot in Figure \ref{spd_fig} the expected empirical excess risk against the sample size for DP-RGD and DP-FM, averaged over 20 runs. We observe a better utility of our proposed method, particular when the sample size is small.

\begin{figure}[t]
\centering
    \subfloat[Leading eigenvector \label{pca_fig}]{\includegraphics[width=0.3\textwidth]{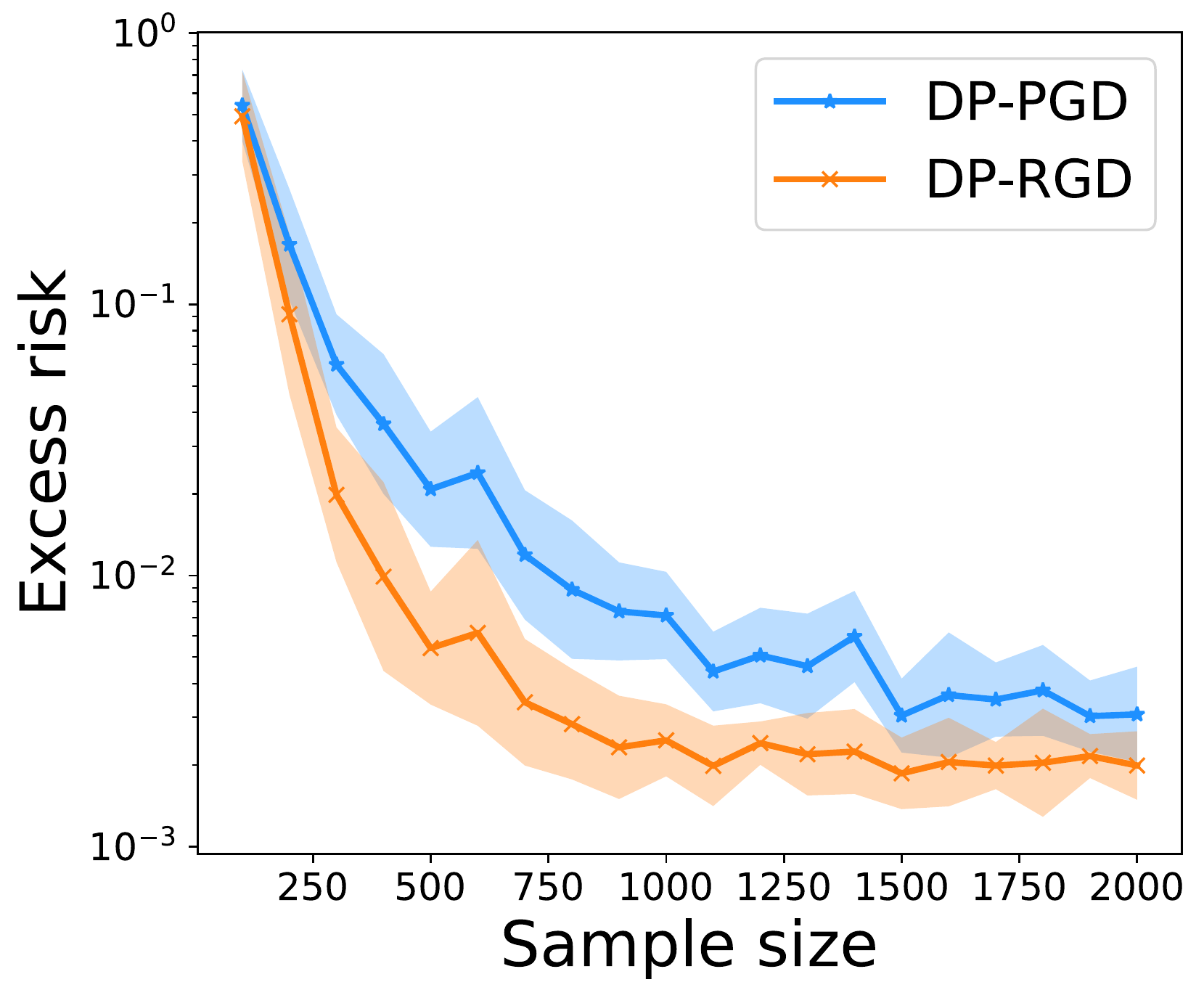}}
    \hspace*{0.3in}
    \subfloat[Fr\'echet mean \label{spd_fig}]{\includegraphics[width = 0.3\textwidth]{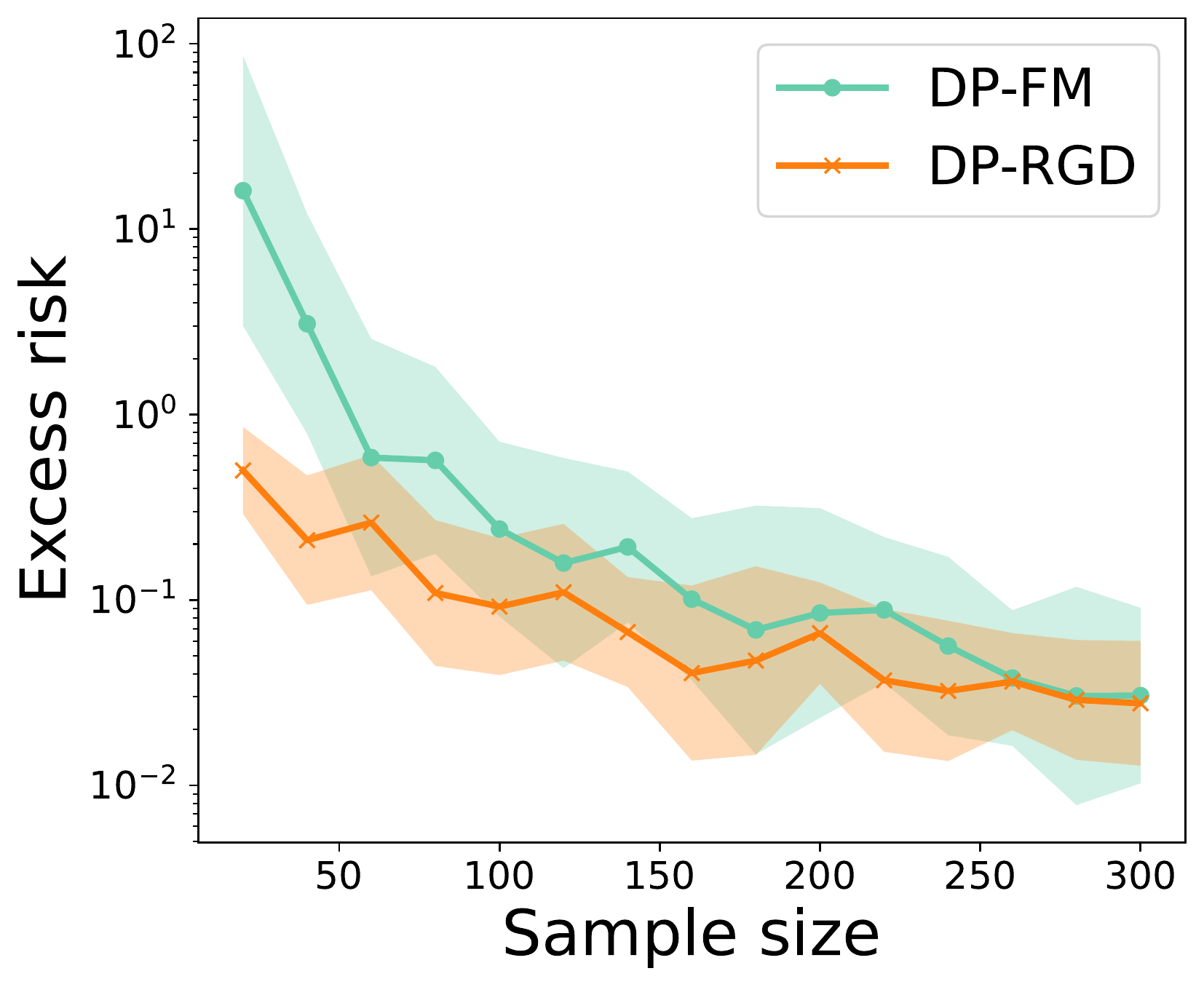}}
\caption{Expected empirical excess risk against sample size on leading eigenvector computation and Fr\'echet mean computation, both averaged over 20 runs. We observe an improved utility guarantees of our proposed DP-RGD for both problems.}
\label{pca_spd_fig}
\end{figure}

\section{Concluding remarks}
We propose a general framework to ensure differential privacy for ERM problems in the Riemannian optimization setting. We develop a general strategy to add noise that adheres to the intrinsic Riemannian geometry. 
To this end, we generalize the Gaussian mechanism to the tangent space compatible with the Riemannian metric. We also prove privacy as well as utility guarantees for a differentially private version of Riemannian (stochastic) gradient descent method. Finally, we highlight that the generalized Gaussian mechanism as well as the analysis toolkit in this paper allows to safeguard differential privacy on Riemannian manifolds beyond the context of Riemannian optimization as long as the operations are defined on tangent space. 


\bibliographystyle{plain}

\appendix

\section{Proofs}

\subsection{Proofs from Section \ref{diff_privacy_manifold_sect}}
\label{proof_for_manifolddp_appendix}

\begin{proof}[Proof of Proposition \ref{gaussian_mech_manifold}]
For any $\xi \sim \gN_w(0, \sigma^2)$, we have $\gR(D) \sim \gN_w(H(D), \sigma^2)$. Set $H(D) - H(D') = \zeta$
and let the privacy loss random variable \cite{dwork2010boosting,dwork2014algorithmic} be defined as 
\begin{align*}
    \gL_{\gR(D) \| \gR(D')} = \log \left( \frac{p(\gR(D) = \gR(D))}{p(\gR(D') = \gR(D))} \right)  &=  \log \left( \frac{\exp(- \frac{1}{2\sigma^2} \| \xi \|_w^2)}{\exp(- \frac{1}{2\sigma^2} \| \xi + \zeta \|^2_w)} \right) \\
    &= \frac{1}{2\sigma^2}  ( 2 \langle \zeta, \xi \rangle_w + \| \zeta \|_w^2).
\end{align*}
By \cite[Lemma 3.17]{dwork2014algorithmic}, it suffices to show $|\gL_{\gR(D) \| \gR(D')}| \leq \epsilon$ with probability $1-\delta$. To see this, we first notice that $\gL_{\gR(D)\|\gR(D')} \sim \gN\big(\frac{\| \zeta \|_w^2}{2\sigma^2}, \frac{\| \zeta \|^2_w}{\sigma^2} \big)$ (by vectorization as in Remark \ref{rmk_Gaussian}). Then we can write $\gL_{\gR(D) \| \gR(D')} =  \frac{\| \zeta \|_w}{\sigma^2} v + \frac{\| \zeta \|^2_w}{2\sigma^2}$, where $v \sim \gN(0, \sigma^2)$. This leads to $|\gL_{\gR(D) \| \gR(D')}| \leq  | \frac{1}{2\sigma^2} ({2\Delta_h} v - {\Delta_h^2} ) |$. The rest of the proof directly follows from \cite[Theorem A.1]{dwork2014algorithmic} by applying the tail bound of Gaussian distribution.
\end{proof}

\begin{proof}[Proof of Lemma \ref{moments_bound_full}]
First, recall that the R\'enyi divergence \cite{renyi1961measures} between two distributions $P, Q$ with parameter $\beta \geq 2$ is defined as ${\rm Div}_\beta(P \| Q) := \frac{1}{\beta - 1} \log \sE_{o \sim P} \big[ (\frac{P(o)}{Q(o)})^{\beta -1} \big]$. Since $\gR(D) \sim \gN_w(H(D), \sigma^2)$, $\gR(D') \sim \gN_w (H(D'), \sigma^2)$, we obtain $K_{\gR, (D, D')}(\lambda) = \lambda {\rm Div}_{\lambda + 1}(\gN_w(H(D), \sigma^2) \| \gN_w(H(D') , \sigma^2))$. By standard results on R\'enyi divergence between multivariate Gaussian distributions, such as \cite{pardo2018statistical}, we have
\begin{equation*}
    K_\gR(\lambda) = \sup_{D \sim D'} \frac{\lambda (\lambda +1)}{2 \sigma^2} \| H(D) - H(D') \|_w^2 \leq \sup_{D \sim D'} \frac{\lambda(\lambda +1)}{2n^2\sigma^2} \| h(z_j) - h(z_j') \|_w^2 \leq \frac{2 \lambda (\lambda + 1) L_0^2}{n^2 \sigma^2},
\end{equation*}
where we assume without loss of generality, $D, D'$ differ only in sample $j$ and the last inequality follows from the Lipschitzness of $h$. 
\end{proof}

\begin{proof}[Proof of Lemma \ref{moments_bound_subsample}]
The proof is a combination of Lemma \ref{moments_bound_full} and the subsampling theorem (\cite[Theorem 9]{wang2019subsampled}) for R\'enyi differential privacy. From Lemma \ref{moments_bound_full} and the definition of R\'enyi differential privacy, we see the mechanism $\gR$ applying to $D_{\rm sub}$ is $(\lambda + 1, \frac{2\lambda L_0^2}{b^2 \sigma^2})$-RDP. Let $\rho(\lambda) = \frac{2\lambda L_0^2}{b^2 \sigma^2}$. Then by the subsampling theorem \cite[Theorem 9]{wang2019subsampled}, we have
\begin{align*}
    K_\gR(\lambda) \leq \log \Big(1 + \frac{b^2}{n^2} \binom{\lambda+1}{2} &\min \Big\{ 4(\exp(\rho(1)) -1 ), 2 \exp(\rho(1)) \Big\} \\
    &+ \sum_{j = 3}^{\lambda+1} \big( \frac{b}{n}\big)^j 2 \exp \big((j-1)\rho(j-1) \big)  \Big).
\end{align*}
By choosing parameters $\sigma^2, b$ and ensuring $\lambda$ is small, it is possible to show $K_\gR(\lambda) \leq  \frac{15 L_0^2 (\lambda+1)}{n^2 \sigma^2}$ as in \cite[Lemma 3.7]{wang2019efficient}, where we use the fact that sensitivity without subsampling is $\frac{4L_0^2}{b^2}$. 
\end{proof}

\subsection{Proofs from Section \ref{diff_private_riemannian_opt_sect}}
\label{proof_for_algorithm_appendix}

\begin{proof}[Proof of Theorem \ref{privacy_guarantee_theorem}]
Let ${\gR}_t(D) = w_{t+1} = {\rm Exp}_{w_t}(-\eta_t \tilde{\gR}_t(D))$, where $\tilde{\gR}_t(D) =\zeta_t  = \frac{1}{b} \sum_{z \in \gB} \grad f(w_t; z) + \epsilon_t$. The final output $\gR(D)$ is a composition of $\gR_T, ..., \gR_1$. First, from the post-processing lemma \cite[Proposition 7.1]{dwork2014algorithmic}, the differential privacy properties of $\gR_t$ is equivalent to that of $\tilde{\gR}_t$. To show privacy guarantees of $\tilde{\gR}_t$, we consider gradient descent and stochastic gradient descent separately. 

Under the setting of gradient descent, i.e., $b = n$, we have $K_{{\gR}_t}(\lambda) \leq \frac{2\lambda (\lambda+1) L_0^2}{n^2\sigma^2}$ by Lemma \ref{moments_bound_full}. Then from the composability theorem \cite[Theorem 2.1]{abadi2016deep}, we have after $T$ iterations, $K_{\gR}(\lambda) \leq \sum_{t=1}^T K_{\gR_t}(\lambda) = \frac{2\lambda(\lambda+1)TL_0^2}{n^2\sigma^2}$. To show $(\epsilon, \delta)$-differential privacy of $\gR$, based on Proposition \ref{prop_rdp_dp}, it is sufficient to show 
\begin{equation*}
    \frac{2\lambda(\lambda+1)TL_0^2}{n^2\sigma^2} \leq \frac{\lambda \epsilon}{2}, \text{ and } \exp \Big(-\frac{\lambda\epsilon}{2} \Big) \leq \delta.
\end{equation*}
Similar as in \cite{abadi2016deep,wang2017differentially}, when $\sigma^2 \geq c_1 \frac{T\log(1/\delta) L_0^2}{n^2 \epsilon^2}$ for some constant $c_1$, we can guarantee the conditions. And hence, the algorithm satisfies $(\epsilon, \delta)$-diffential privacy.

Under the setting of stochastic gradient descent, i.e., $b < n$, we can bound $K_{\gR_t}(\lambda)$ as in Lemma \ref{moments_bound_subsample}. 
After $T$ iterations, the cumulative bound is $\frac{15T(\lambda+1)L_0^2}{n^2\sigma^2}$ under sufficiently large $\sigma^2$. Based on similar reasoning as in the full gradient case, we see when $\sigma^2 \geq c_2 \frac{ T\log(1/\delta)L_0^2}{n^2\epsilon^2}$, the algorithm is $(\epsilon, \delta)$-differentially private.
\end{proof}

\begin{proof}[Proof of Lemma \ref{lemma_bound}]
First, using unbiasedness of $\grad f(w;z)$, we have $\sE[\zeta] = \grad F(w)$. Consider the vectorization $\vec{\epsilon}$ and the tangent space Gaussian distribution expressed as a standard multivariate Gaussian. Thus, $\sE \| \epsilon\|_w^2 = \sE \| \vec{\epsilon} \|^2_2 = \sigma^2 \trace(G_w^{-1}) \leq d \sigma^2 c_l^{-1}$. Finally, we have
\begin{equation*}
    \sE \| \zeta \|_w^2 = \sE \| \frac{1}{b} \sum_{z \in \gB} \grad f(w; z)\|_w^2 + \sE \| \epsilon \|^2_w \leq L_0^2 + d\sigma^2 c_l^{-1},
\end{equation*}
where we notice $\sE [\langle \grad f(w, z) , \epsilon \rangle_w] = 0$ as the noise is independent.
\end{proof}

\begin{proof}[Proof of Theorem \ref{g_convex_theorem}]
The proof is adapted from the proof of \cite[Theorem 10]{zhang2016first}. Since each $f(w; z)$ is geodesic convex, then $F(w) = \frac{1}{n} \sum_{i=1}^n f(w; z_i)$ is geodesic convex. Then from the first-order characterization of geodesic convex functions \cite{boumal2020introduction}, we have $F(w_t) - F(x^*) \leq -   \langle \grad F(w_t), {\rm Exp}_{w_t}^{-1}(w^*) \rangle_{w_t}$. Further by Lemma \ref{curvature_lemma} on a geodesic triangle formed by $w_t, w_{t+1}, w^*$, we obtain
\begin{align}
    \sE_t[d^2(w_{t+1}, w^*)] -  d^2(w_t, w^*) &\leq \varsigma \eta_t^2 \sE_t \| \zeta_t \|^2_{w_t} + 2 \eta_t \langle \grad F(w_t), {\rm Exp}^{-1}_{w_t}(w^*) \rangle_{w_t}, \nonumber \\
    &\leq  \varsigma \eta_t^2 (L_0^2 + d\sigma^2 c_l^{-1}) + 2 \eta_t \langle \grad F(w_t), {\rm Exp}^{-1}_{w_t}(w^*) \rangle_{w_t} \label{eq1_g_convex}
\end{align}
where we use Lemma \ref{lemma_bound} and notice that ${\rm Exp}^{-1}_{w_t}(w_{t+1}) = - \eta_t \zeta_t$. The expectation is over both the randomness in the noise and the subsampling (if $b < n$) at iteration $t$. 
Combining \eqref{eq1_g_convex} and the first order characterization yields
\begin{align*}
    F(w_{t}) - F(w^*) &\leq  - 2 \eta_t \langle \grad F(w_t), \sE[ \zeta_t ] \rangle_{w_t}\\
    &\leq \frac{1}{2\eta_t} \Big( d^2(w_t, w^*) - \sE_t [d^2(w_{t+1}, w^*)] \Big) + \frac{\varsigma \eta_t (L_0^2 + d \sigma^2 c_l^{-1})}{2}
\end{align*}
where the expectation is over the randomness of both the noise and subsampling (if $b < n$). Telescoping this inequality from $t = 0, ..., T-1$ and setting $\eta_t = \frac{D_\gW}{\sqrt{(L_0^2 + d \sigma^2 c_l^{-1}) \varsigma T}}$ gives
\begin{align*}
    \frac{1}{T} \sum_{t = 0}^{T-1} \sE[F(w_t)] - F(w^*) &\leq \frac{\sqrt{(L_0^2 + d \sigma^2 c_l^{-1}) \varsigma}}{2 \sqrt{T} D_\gW} \sE \Big( d^2(w_0, w^*) - d^2(w_T, w^*) \Big) + \frac{\varsigma \eta_t (L_0^2 + d \sigma^2 c_l^{-1})}{2} \\
    &\leq D_\gW \sqrt{\frac{(L_0^2 + d \sigma^2 c_l^{-1}) \varsigma}{T}} 
\end{align*}
where we take expectation over all iterations. Finally, it can be shown from the property of geodesic averaging that $F(w^{\rm priv}) \leq \frac{1}{T} \sum_{t=0}^{T-1} F(w_t)$.
\end{proof}

\begin{proof}[Proof of Theorem \ref{g_strongly_convex_theorem}]
The proof adapts from \cite[Theorem 12]{zhang2016first} and follows similarly as in Theorem \ref{g_convex_theorem}. 
\end{proof}

\begin{proof}[Proof of Theorem \ref{PL_condition_utility_theorem}]
First consider the gradient descent with $\zeta_t = \grad F(w_t) + \epsilon_t$. Then by geodesic $L_1$-smooth, we have 
\begin{align*}
    \sE_{\epsilon_t} [F(w_{t+1})] &\leq F(w_t) - \eta_t \| \grad F(w_t) \|^2_{w_t} + \frac{L_1 \eta_t^2}{2} \sE_{\epsilon_t} \| \grad F(w_t) + \epsilon_t \|^2_{w_t} \\
    &\leq F(w_t) - (\eta_t - \frac{L_1 \eta_t^2}{2}) \| \grad F(w_t) \|_{w_t}^2 + \frac{L_1 \eta_t^2 d \sigma^2 c_l^{-1}}{2} \\
    &\leq F(w_t) - \frac{\eta_t}{2} \| \grad F(w_t) \|_{w_t}^2 + \frac{L_1 \eta^2_t d \sigma^2 c_l^{-1}}{2 },
\end{align*}
where the last inequality uses $\eta_t \leq 1/L_1$. From the PL condition, we have 
\begin{equation*}
    \sE_{\epsilon_t}[F(w_{t-1})] - F(w^*) \leq \Big(1 - \frac{\tau \eta_t }{2} \Big)( F(w_t) - F(w^*) ) + \frac{L_1 \eta^2_t d \sigma^2 c_l^{-1}}{2 }.
\end{equation*}
Applying this result recursively by choosing $\eta_t = \eta < \min\{ 1/L_1, 1/\tau \}$ and taking full expectation yields
\begin{align*}
    \sE[F(w_T))] - F(w^*) &\leq \Big( 1 - \frac{\tau\eta}{2} \Big)^T \big( F(w_0) - F(w^*) \big) + \frac{L_1 \eta^2 d \sigma^2 c_l^{-1}}{2 } \sum_{t = 0}^T (1 - \frac{\tau \eta}{2})^t \\
    &\leq \Big( 1 - \frac{\tau\eta}{2} \Big)^T \big( F(w_0) - F(w^*) \big) + \frac{L_1 \eta^2 d \sigma^2 c_l^{-1}}{2 } \sum_{t = 0}^\infty (1 - \frac{\tau \eta}{2})^t \\
    &\leq \Big( 1 - \frac{\tau\eta}{2} \Big)^T \big( F(w_0) - F(w^*) \big) + \frac{d \sigma^2 c_l^{-1}}{2 \tau} 
\end{align*}
where the last inequality follows from the limit of a geometric series and $\eta < \frac{1}{L_1}$. Finally, choosing $T = \log(\frac{n^2\epsilon^2}{d L_0^2 \log(1/\delta) c_l^{-1}})$ gives the desired result. 

Now we consider stochastic gradient descent with $\zeta_t = \frac{1}{b}\sum_{z \in \gB_t} \grad f(w_t; z) + \epsilon_t$. Similarly by geodesic $L_1$-smooth, we obtain
\begin{align*}
    &\sE_{\gB_t, \epsilon_t}[F(w_{t+1})] \\
    &\leq F(w_t) - \eta_t \| \grad F(w_t) \|^2_{w_t} + \frac{L_1 \eta_t^2}{2} \sE_{\gB_t, \epsilon_t}\| \frac{1}{b} \sum_{z \in \gB_t} \grad f(w_t; z) + \epsilon_t \|_{w_t}^2 \\
    &\leq F(w_t) - \eta_t \| \grad F(w_t) \|^2_{w_t} + \frac{L_1 \eta_t^2}{2} \sE_{\gB_t} \| \frac{1}{b} \sum_{z \in \gB_t} \grad f(w_t; z) \|^2_{w_t} + \frac{L_1 \eta_t^2 d \sigma^2 c_l^{-1}}{2} \\
    &\leq F(w_t) -  \tau \eta_t ( F(w_t) - F(w^*) ) + \frac{L_1 L_0^2 + L_1 d\sigma^2 c_l^{-1} }{2} \eta_t^2.
\end{align*}
Let $\eta_t = \eta < \min\{ 1/L_1, 1/\tau \}$, and apply the result recursively. Then we have 
\begin{align*}
    \sE[F(W_T)] - F(w^*) &\leq (1 - \tau \eta)^T (F(w_0) - F(w^*)) + \frac{L_1L_0^2 + L_1 d \sigma^2 c_l^{-1}}{2} \eta^2 \sum_{t=0}^T (1 - \tau \eta)^t \\
    &\leq (1 - \tau \eta)^T (F(w_0) - F(w^*)) + \frac{L_0^2 + d \sigma^2 c_l^{-1}}{2\tau} \\
    &\leq O \Big( \frac{\tau^{-1} d \log(n) \log(1/\delta) c_l^{-1} L_0^2 )}{n^2\epsilon^2} (F(w_0) - F(w^*)) \Big)
\end{align*}
under the same choice of $T$. This matches the bound under full gradient setting. 
\end{proof}

\begin{proof}[Proof of Theorem \ref{general_nonconvex_theorem}]
First consider gradient descent with $\zeta_t = \grad F(w_t) + \epsilon_t$. Then by geodesic $L_1$-smooth, we obtain
\begin{align*}
    \sE_{\epsilon_t}[F(w_{t+1})] &\leq F(w_t) - (\eta_t - \frac{L_1\eta_t^2}{2}) \| \grad F(w_t) \|^2_{w_t} + \frac{L_1 \eta_t^2 d \sigma^2 c_l^{-1}}{2} \\
    &= F(w_t) - \frac{1}{2L_1} \| \grad F(w_t) \|_{w_t}^2 + \frac{d \sigma^2 c_l^{-1}}{2 L_1}
\end{align*}
where we choose $\eta_t= \eta = 1/L_1$. Telescoping this inequality for $t = 0, ..., T-1$ and taking expectation yields
\begin{align*}
    \frac{1}{T} \sum_{t=0}^{T-1} \sE \| \grad F(w_t) \|^2_{w_t}  &\leq \frac{2 L_1}{T} (F(w_0) - \sE[F(w_T)]) + {d\sigma^2 c_l^{-1}} \\
    &\leq O \Big( \frac{\sqrt{d L_1 \log(1/\delta) c_l^{-1}} L_0}{n\epsilon} \Big)
\end{align*}
by the choice of $\sigma^2$ and $T = \frac{\sqrt{L_1} n \epsilon}{\sqrt{d \log(1/\delta) c_l^{-1}}L_0}$. Given the output $w^{\rm priv}$ is uniformly selected from $\{ w_0, ..., w_{T-1}\}$, we have $\sE\| \grad F(w^{\rm priv}) \|^2 = \frac{1}{T} \sum_{t=0}^{T-1} \sE \| \grad F(w_t) \|^2_{w_t}$. 

Similarly, under stochastic setting, $\zeta_t = \frac{1}{b}\sum_{z \in \gB_t} \grad f(w_t; z) + \epsilon_t$ and 
\begin{align*}
    \sE_{\epsilon_t}[F(w_{t+1})] \leq 
    F(w_t) - \frac{1}{L_1}  \| \grad F(w_t) \|^2_{w_t} + \frac{ L_0^2 + d \sigma^2 c_l^{-1}}{2 L_1},
\end{align*}
where we use the fact that $\eta_t = \frac{1}{L}$. Following the same argument and choice of $\sigma^2$ and $T$, we achieve the same bound as full gradient case.
\end{proof}

\end{document}